\documentclass[12pt]{article}
\usepackage{amsfonts, amsmath, amssymb, amsthm, enumitem}  
\usepackage[english]{babel}
\usepackage[utf8]{inputenc}
\usepackage{bm}
\usepackage{tikz}
\usetikzlibrary{patterns,decorations.text,positioning,arrows,shapes,decorations.markings}

\usepackage{geometry}
\usepackage{textcomp, cmap, comment}	
\usepackage{graphicx, wrapfig}
\usepackage{epigraph, xcolor, hyperref}

\newcommand{\R}{{\mathbb R}}

\newcommand{\aaa}{{\mathcal A}}

\newcommand{\ff}{\mathcal F}
\theoremstyle{plain}
\newtheorem{thm}{Theorem}

\newtheorem{cla}{Claim}

\newtheorem{pro}{Problem}

\theoremstyle{definition}

\title{Intersection theorems for triangles}
\author{Peter Frankl,\footnote{Renyi Institute, Budapest, Hungary and Moscow Institute of Physics and Technology, Russia.} Andreas Holmsen,\footnote{Department of Mathematical Sciences, KAIST, Daejeon, South Korea.  Research supported by NRF grant No.~2020R1F1A1A0104849011.}
Andrey Kupavskii\footnote{G-SCOP, CNRS, University Grenoble-Alpes, France and Moscow Institute of Physics and Technology, Russia; Email: {\tt kupavskii@yandex.ru}.  The first and the third author acknowledge the financial support from the Russian Government in the framework of MegaGrant no 075-15-2019-1926.}}
\date{}

\begin{document}
\maketitle
\begin{abstract}
    Given a family of sets on the plane, we say that the family is {\it intersecting} if for any two sets from the family their interiors intersect. In this paper, we study intersecting families of triangles with vertices in a given set of points. In particular, we show that if a set $P$ of $n$ points is in convex position, then the largest intersecting family of triangles with vertices in $P$ contains at most $(\frac{1}{4}+o(1))\binom{n}{3}$ triangles. 
\end{abstract}
Let $P$ be a set of $n$ points in general position in the plane (that is, no three points are collinear). A triangle is {\em spanned} by $P$ if its vertices belong to $P$, and we say that two triangles spanned by $P$ are {\em intersecting} if their interiors intersect. In this note, we study {\em intersecting families} of triangles spanned by a point set $P$, where a family is  intersecting if any two of its members are intersecting. 

\paragraph{Background.}
It is a well-known result due to Boros and F\"uredi \cite{boros} that for every set $P$ of $n$ points in the plane in general position there exists a point of depth at least  $\frac{2}{9} \binom{n}{3}$, 
where the {\it depth} of a point $x$ is the number of triangles spanned by $P$ that contain $x$ in their interior. 
For an alternative proof of this fact see Bukh \cite{bukh}, and for a construction of a point set for which this bound is attained see Bukh, Matou\v sek, Nivasch \cite{nivasch}. 
In general, it is known that there always exists a point of depth at least $\alpha \binom{n}{d+1}$ for a set of $n$ points in $\R^d$, 
where $\alpha=\alpha(d)>0$ is a function depending only on $d$. 
This result was originally proved by B\'ar\'any \cite{barany} who showed the bound $\alpha>d^{-d}$, 
while the best known bound $\alpha \geq \frac{2d}{(d+1)!(d+1)}$ is due to Gromov \cite{gromov}. 
(See Karasev \cite{karasev} for a simple proof of Gromov's lower bound.) 

At the same time, we know the exact upper bound for the maximum depth of a point. For an integer $n\geq 3$, put $F(n):= \binom{\lceil\frac{n+2}{2}\rceil}{3} + \binom{\lfloor\frac{n+2}2\rfloor}{3}$.
Note that $F(n)$ is strictly increasing and $\lim_{n\to \infty} \frac{F(n)}{\binom{n}{3}} = \frac{1}{4}$.
It is a known fact that if $P$ is a set of $n$ points in the plane, then no point is contained in the interior of more than $F(n)$ triangles spanned by $P$. This can be deduced from the ``upper bound theorem'' for convex polytopes using Gale duality (see 
\cite[Remark 4.2]{uli}).

Another important motivation for our studies is the celebrated Erd\H os--Ko--Rado theorem \cite{EKR}. 
It states that a family of $k$-element subsets of an $n$-element set in which any two sets intersect has size at most $\binom{n-1}{k-1}$, provided $n\ge 2k$. 
Importantly, $\binom{n-1}{k-1}$ is the size of the family of all sets containing a fixed point, which trivially is an intersecting family (any subfamily of such a family is called a {\it trivially intersecting family}). 
A similar phenomenon occurs in many other settings, such as permutations \cite{perms}, vector spaces \cite{hsieh, vector} etc.:
{\em the largest intersecting substructure is the trivial one.} 
The goal for this note is to initiate the exploration of this phenomenon in a geometric setting.

\paragraph{Results.} 
Our first theorem deals with a seemingly restricted setting of intersecting families of triangles spanned by the vertices of a regular $n$-gon, $n\geq 3$. Let $S^1$ denote the unit circle centered at the origin in $\mathbb{R}^2$ and $K_n$ be a regular $n$-gon inscribed in $S^1$. Let $x$ be a point very close to the center of $S^1$ which is in general position with respect to the vertices of $K_n.$


\begin{thm}\label{thm1}
Let $V$ be the vertex set of $K_n$. Then the size of the largest intersecting family of  triangles spanned by $V$ is equal to the number of triangles spanned by $V$ that contain $x$.
\end{thm}

It is not difficult to verify that the number of triangles spanned by $V$ that contain $x$ is equal to $F(n)$. We do so in the beginning of Section~\ref{proof1}. 

Now let us take any set $P$ of $n$ points in convex position and let $\ff$ be an intersecting family of triangles spanned by $P$. There is a natural cyclic order on $P$, and whether or not two triangles spanned by $P$ intersect depends only on the relative positions of their vertices with respect to this cyclic ordering. Now map the points of $P$ to the vertices $V$ of $K_n$ by an order preserving map. Thus $\mathcal{F}$ is mapped to an intersecting family $\mathcal{F}'$ spanned by $V$. Therefore, Theorem \ref{thm1} implies the following theorem.

\begin{thm}\label{thm3}  
Let $P$ be a set of $n\geq 3$ points in convex position in the plane. Then no intersecting family of triangles spanned by $P$ can contain more than $F(n)$ triangles. On the other hand, there exists an intersecting family of triangles spanned by $P$ that has exactly $F(n)$ triangles.
\end{thm}

In the other direction, Theorem \ref{thm1} (together with the remark after the theorem) clearly follows from Theorem \ref{thm3}. 
In Sections~\ref{proof1} and~\ref{proof3}, we will give two proofs of Theorem \ref{thm1} (one by double-counting and one using a certain inductive procedure on the family), which do not depend on the knowledge of the function $F(k)$. In Section~\ref{proof2}, we will give yet another proof of Theorem~\ref{thm3} by induction on $n$.

The analogous question for general point sets is very interesting and is discussed in the last section.   

Interestingly, it is not always the case that the largest intersecting family of triangles in a point set is a trivial one. Indeed, an easy conclusion of Theorem~\ref{thm1} and the two paragraphs that come after is that, whenever we have a set of $n$ points in convex position, it contains an intersecting family of triangles $\mathcal F$ of size exactly $F(n)$: this family is simply the `image' of the family of triangles in $K_n$ that contain $x.$ In particular, this shows that Theorem~\ref{thm3} is tight. However, this family need not be a trivial family. Indeed, the set of points constructed by Bukh, Matou{\v s}ek, and Nivasch \cite[Theorem 1.3]{nivasch} is in convex position, 
and the maximum depth of a point in this set is $\big(\frac{2}{9}+o(1)\big) \binom{n}{3}$.\footnote{This set of points consists of three `clusters' of points, each located aroud a vertex of an equilateral triangle}

\vspace{1ex}

Our last theorem gives a continuous counterpart of the statement of Theorem~\ref{thm3}.
Consider the set $\mathfrak{T} = S^1\times S^1\times S^1$ of all ordered triples of points inscribed in the unit circle $S^1$. 
We associate 
a triple $(p,q,r) \in \mathfrak T$ with the triangle spanned by the points $p,q,r$. (Some triangles will be degenerate, but the set of degenerate triangles has measure 0).
By a simple application of Theorem \ref{thm3} we have the following. 

\begin{thm}\label{thm4}
 Let $\nu$ be a probability measure on $S^1$ and let $\mu_s = \nu\times \nu\times \nu$ be the corresponding product measure on $\mathfrak T$.  Let $\ff\subset \mathfrak T$ be an intersecting measurable family of triangles. Then $\mu_s(\ff)\le \frac 14$.
\end{thm}
\begin{proof}
Fix an intersecting family $\ff\subset \mathfrak T$. Let us generate a random triangle from $\mathfrak T$ as follows: first, choose a set $P$ of $n$ points on $S^1$, where each point is chosen independently according to the distribution $\nu$. Second, choose a triple of points from $P$ uniformly at random from the set of all triples. It is easy to see that the obtained distribution on $\mathfrak T$ is the same as the distribution $\mu_s$.

Using Theorem~\ref{thm3}, any intersecting family of triangles on $P$ has measure at most $\frac{1}{4} +o(1)$ as $n\to \infty$. Thus,  $\big|\ff\cap \binom{P}{3}\big|\le (\frac{1}{4}+o(1)) \binom{n}{3}.$ Since this inequality is valid for any set $P$ of $n$ points, the theorem follows. 
\end{proof}

\section{Proofs}

\subsection{First proof of Theorem 1} \label{proof1} 
In this subsection, all edges and triangles are spanned by $V$, 
that is, their endpoints (vertices) are in $V$. 

Let us start by showing that the number of triangles containing $x$ is $F(n).$ First consider the case of odd $n$.  Choose the first vertex $v$ of the triangle in $n$ ways. Then the line $vx$ has $\frac{n-1}{2}$ points on each side. 
In order to contain $x$ in the interior, (i) the remaining vertices $v',v''$ of the triangle must come from different sides of the line $vx$; (ii) out of the two arcs connecting $v'$ and $v''$ the one  containing $v$ must be longer. 
There are $i+1$ choices of $v',v''$ so that the shorter arc has exactly $i$ vertices of $K_n$ in the interior, 
and $i\leq \frac{n-3}{2}$ by condition (ii).
Summing over $i$ (and dividing by $3$ because we count each triangle $3$ times), we get that the number of triangles containing $x$ equals
\[\textstyle{\frac {n}{3}  \big( 1+2 + \dots + \frac{n-1}{2} \big) = \frac{n}{3} \binom{\frac{n+1}{2}}{2} = \frac{(n-1)n(n+1)}{24} = \frac{(n-1)(n-3)(n+1)}{48} + \frac{(n-1)(n+3)(n+1)}{48} = F(n).}\]

If $n$ is even then we do the same count, but first excluding the triangles with one side being the diameter of $S^1.$ The number of such triangles equals $\frac{n}{3} \big( 0+1+\ldots+\frac{n-4}{2} \big) = \frac{n(n-2)(n-4)}{24}$. Next, out of triangles with one side being the diameter exactly half of them contain $x$, which equals $\frac {n(n-2)}4.$ Summing up, we conclude that the number of triangles containing $x$ is $\frac{n(n-2)(n+2)}{24} = F(n)$.\\

We now start the proof of Theorem \ref{thm1}. First some definitions. Let $e_1$ and $e_2$ be a pair of edges that share a common vertex. In other words, $e_1$ and $e_2$ are the edges of a path on three vertices. (See Figure \ref{fig1}, left.) We add edges parallel to $e_1$ or $e_2$ to the endpoints of the current path alternately to obtain a longer path. We repeat this process until we reach a unique maximal path called the {\em path defined by  $e_1$ and $e_2$}. (See Figure \ref{fig1}, center.) 
Taking the convex hull of consecutive pairs of edges on the path defined by $e_1$ and $e_2$, we get a family of triangles that we call {\em the strip of $e_1$ and $e_2$}. The strip of $e_1$ and $e_2$ is {\em nontrivial} if it consists of at least two triangles. Note that the union of all triangles from one strip is a triangulation of a convex polygon. In particular, the triangles from one strip have pairwise disjoint interiors. (See Figure \ref{fig1}, right.)
The endpoints of the path $e_1e_2$ split the circle into two open arcs. Out of those two, denote by $\gamma$ the {\it open} arc that does not contain the midpoint of the path $e_1e_2$. Finally,  let $\text{step}(e_1, e_2)$ denote the number of vertices of $V$ lying on $\gamma$.

\begin{figure}
  \centering
\begin{tikzpicture}[scale = 0.43]

\begin{scope}
\filldraw[red] (0,0) circle (3pt);
\draw (0,0) circle  (4);
\draw[thick,blue] ({360/21}:4) -- ({360*4/21}:4) -- ({360*18/21}:4);
\foreach \x in {0,...,20}
\filldraw ({360*\x/21}:4) circle (3pt);
\end{scope}

\begin{scope}[xshift = 12cm]
\filldraw[red] (0,0) circle (3pt);
\draw (0,0) circle  (4);
\draw[thick,blue] ({360/21}:4) -- ({360*4/21}:4) -- ({360*18/21}:4) --  ({360*8/21}:4) -- ({360*14/21}:4) -- ({360*12/21}:4);
\foreach \x in {0,...,20}
\filldraw ({360*\x/21}:4) circle (3pt);
\end{scope}

\begin{scope}[xshift = 24cm]
\draw (0,0) circle  (4);
\draw[thick,blue] ({360/21}:4) -- ({360*4/21}:4) -- ({360*18/21}:4) --  ({360*8/21}:4) -- ({360*14/21}:4) -- ({360*12/21}:4);

\fill[blue, opacity = .25] 
({360/21}:4) -- ({360*4/21}:4) -- ({360*18/21}:4)
({360*4/21}:4) -- ({360*18/21}:4) --  ({360*8/21}:4)
({360*18/21}:4) --  ({360*8/21}:4) -- ({360*14/21}:4)
({360*8/21}:4) -- ({360*14/21}:4) -- ({360*12/21}:4)
;
\foreach \x in {0,...,20}
\filldraw ({360*\x/21}:4) circle (3pt);
\filldraw[red] (0,0) circle (3pt);

\end{scope}
\end{tikzpicture}
  \caption{Definition of a strip. On the left: edges $e_1$ and $e_2$ sharing an endpoint. In the center: the path defined by $e_1$ and $e_2$, where $\text{step}(e_1,e_2) = 3$. On the right: the  strip of $e_1$ and $e_2$.}\label{fig1}
\end{figure}
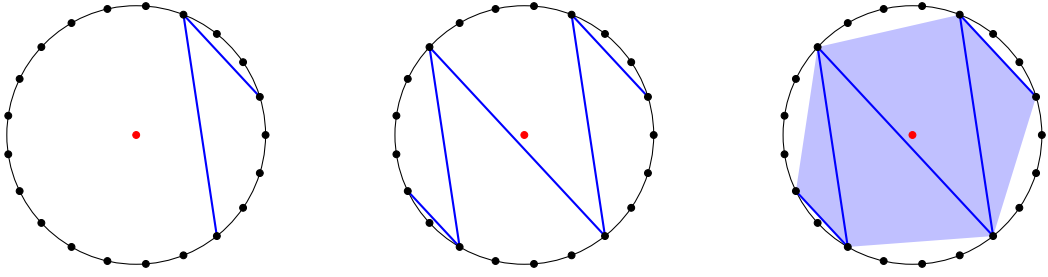

\begin{cla}\label{cla1} In each nontrivial strip there is exactly one triangle that contains $x$ in its interior.
\end{cla}
\begin{proof} Consider a pair of edges $e_1$ and $e_2$ that share a common vertex. Let $a_1$ denote the number of vertices of $V$ that lie on the open circular arc bounded by $e_1$ on the side opposite to $e_2$. Define $a_2$ analogously. The $2$-path $e_1e_2$ can be extended to a nontrivial strip if and only if $\max\{a_1,a_2\} > \text{step}(e_1, e_2)$.
Thus, \[|V| = n = a_1 +a_2 +\text{step}(e_1, e_2) +3 > 2\cdot\text{step}(e_1,e_2) + 3.\]
We conclude that  if $e_1$ and $e_2$ define a nontrivial strip, then $\text{step}(e_1, e_2) < \big\lfloor \frac{n-3}2 \big\rfloor$. 

We now show that every nontrivial strip contains $x$. Let $W$ be the vertices of the path defined by $e_1$ and $e_2$. If we delete the points of $W$ from the unit circle, then the remaining set  $S^1\setminus W$ is a disjoint union of open circular arcs. By the maximality of the path determined by $e_1$ and $e_2$, each of these open arcs contains at most $\text{step}(e_1,e_2)$ points from $V \setminus W$. If $x$ is not contained in the the convex hull of $W$, then the center can be separated from $W$ by some straight line. The side of this line opposite to the one containing $W$ contains at least $\lfloor\frac {n-2}2\rfloor$ points from $V\setminus W$ (since $V$ is uniformly distributed and $x$ is sufficiently close to the origin). But this contradicts the observation that each open arc of $S^1\setminus W$ contains at most $\text{step}(e_1,e_2)<\lfloor \frac{n-3}2 \rfloor$ points. This shows that the union of the triangles of a nontrivial strip always contains the center. Clearly, the center must be contained in the interior of exactly one of its triangles since the union of the triangles of a strip form a triangulation of a convex polygon. (Note that, by our choice of $x$, it is not on the boundary of any triangle spanned by $V$.)
\end{proof}

\begin{cla}\label{cla2}
Every triangle spanned by $V$ is contained in at most two nontrivial strips. Moreover, any triangle not containing $x$ is contained in exactly two nontrivial strips.
\end{cla}
\begin{proof}
Clearly, every triangle is contained in three strips, one for each of the three possible pairs of its edges. We need to show that at least one of these is trivial. Let $e_1, e_2, e_3$ be the edges of a triangle, and let $a_1\ge a_2\ge a_3$ denote the numbers of points on the open arcs bounded by the three sides $e_1, e_2, e_3$, respectively. Recall that, for $\{i_1,i_2,i_3\} = \{1,2,3\}$, the $2$-path $e_{i_1}e_{i_2}$ can be extended to a nontrivial strip if and only if $\max\{a_{i_1},a_{i_2}\} > a_{i_3}$. Since $a_1\ge \max\{a_2,a_3\}$, it should be clear that the strip defined by the pair $e_2$ and $e_3$ consists of only one triangle. 
This proves the first part of the statement.

To show the validity of second part, it is enough to notice that if a triangle does not contain $x$  then, in the notation above, we have $a_1> a_2 \geq a_3$, and thus we can produce a non-trivial strip starting from both $e_1,e_2$ and $e_1,e_3$. 
\end{proof}

The proof of the theorem is now concluded by a double-counting argument. 
Let $\ff$ be an intersecting family of triangles spanned by $V$.  
Then, clearly, for each non-trivial strip $\mathcal S$ we have $|\mathcal S\cap \ff| \le 1$ since the  triangles in each strip have pairwise disjoint interiors. 
Let $\mathcal C$ denote the family of triangles containing $x$ and let $\mathcal W$ denote the collection of all non-trivial strips. Using Claim~\ref{cla1}, we get
\begin{equation}\label{eq1} \sum_{\mathcal S\in \mathcal W}|\mathcal S\cap \ff|\le |\mathcal W|=\sum_{\mathcal S\in \mathcal W}|\mathcal S\cap \mathcal C|.\end{equation}

We now express the two sums in \eqref{eq1} in different ways. By Claim~\ref{cla2}, the family $\mathcal C$ can be partitioned as $\mathcal{C} = \mathcal{C}_0 \cup \mathcal{C}_1 \cup \mathcal{C}_2$, 
where $\mathcal{C}_i$ denotes the subfamily of triangles contained in precisely $i$ non-trivial strips. 
We have \begin{equation}\label{eq123} \sum_{\mathcal S\in \mathcal W}|\mathcal S\cap \ff| = 2|\ff\setminus (\mathcal C_1\cup \mathcal C_0)|+|\ff\cap \mathcal C_1|.\end{equation} 
At the same time, we have $\sum_{\mathcal S\in \mathcal W}|\mathcal S\cap \mathcal C| = 2|\mathcal C_2|+|\mathcal C_1|.$
Plugging this into \eqref{eq1}, we get
\begin{equation*}
  2|\ff| \overset{\eqref{eq123}}{=} \sum_{\mathcal S\in \mathcal W}|\mathcal S\cap \ff|+|\ff\cap \mathcal C_1|+2|\ff\cap \mathcal C_0|\le \sum_{\mathcal S\in \mathcal W}|\mathcal S\cap \mathcal C|+|\mathcal C_1|+2|\mathcal C_0| = 2|\mathcal C|,
  \end{equation*}
\subsection{Second proof of Theorem 1} \label{proof3} 
Consider an intersecting family $\mathcal F$ spanned by $V$. We are going to inductively modify $\mathcal F$ by deleting and adding triangles such that
(i) at each step the family stays intersecting and (ii) at each step we add at least as many triangles as we delete. The process will terminate when we reach a subfamily of the family $\mathcal C$ of all triangles containing $x$.

If $\mathcal F \subset \mathcal C$ then we are done, so we may assume $\mathcal{F} \not\subset \mathcal{C}$, which means that at least one triangle in $\mathcal{F}$ does not contain $x$. Let $T\in \mathcal F$ be a triangle that has the largest distance from $x$ and let $ab$, where $a,b\in K_n$, be the side of $T$ that is closest to $x$. 
Let $\mathcal F_{ab}\subset \mathcal F$ denote the subfamily of all triangles in $\mathcal F$ that have side $ab$. Note that every triangle in $\mathcal{F}\setminus\mathcal{F}_{ab}$ contains a point $y$ in the relative interior $\mbox{relint}(ab)$ of $ab$.
This follows from the assumption that $\mathcal{F}$ is intersecting and since $T$ maximizes the distance from $x$. 
We define a new family $\mathcal F' = (\mathcal F\setminus \mathcal F_{ab})\cup \mathcal G_{ab},$ where $\mathcal G_{ab}$ is the 
family of all triangles that have side $ab$ and do not intersect the interior of $T$ (i.e., their third vertex is lying on the same side of $ab$ as $x$). 

We make two simple observations. First, there are at least as many points of $K_n$ on the side of $ab$ that contains $x$ as there points of $K_n$ on the side containing the third vertex of $T$. Thus, $|\mathcal F'|\ge |\mathcal F|.$ Second, we observe that $\mathcal F'$ is intersecting. Indeed, the triangles in $\mathcal G_{ab}$ all share an interior point, say, close to the midpoint of $ab$. Moreover, as we noted above, any triangle in $\mathcal{F} \setminus \mathcal{F}_{ab}$ contains some point $y \in \mbox{relint}(ab)$ in its interior, and thus contains any point sufficiently close to $y$ in its interior. But any point sufficiently close to $y$ and lying on the same side of $ab$ as $x$ is also contained in the interior of any triangle in $\mathcal F_{ab}'.$

By performing this replacement procedure sufficiently many times, the initial family $\mathcal F$ eventually transforms into a subfamily of $\mathcal C$. Indeed, it is not difficult to see that, for any triangle in $\mathcal G_{ab}$ the side $ab$ is not the closest to $x$. Due to this and the extremeness of the choice of $ab$, the side $ab$ cannot appear twice in this process, and so the number of replacement steps is at most ${n\choose 2}$. At the same time, the size of the family does not decrease at any step. This shows that $|\mathcal F|\le |\mathcal C|.$
  
\subsection{Proof of Theorem~\ref{thm3}} \label{proof2}

Now we give a proof of Theorem \ref{thm3}, which also provides yet another proof of Theorem \ref{thm1}. We have shown in the introduction that there is an intersecting family of size $F(n)$, and thus we only need to show the upper bound.
First observe that the theorem holds holds trivially for $n=3$, and is easily verified for $n=4$. We will proceed by induction on $|P|$. Let $P$ be a set of $n>4$ points in convex position, and let $\ff$ be an intersecting family of triangles spanned by $P$. 

Recall that we have a natural cyclic order on $P$ induced by convex position. For the proof, an {\em arc} is a subset of $P$ that consists of consecutive points in this order. For two points $p_1,p_2\in P$, we denote by $A_{p_1p_2}$ the arc with endpoints $p_1,p_2$ and that goes clockwise from $p_1.$ We also use the following convention: the vertices of any triangle $pqr$ ($p,q,r\in P$) used in this proof are given in clockwise order, that is, we meet $q$ before $r$ when going clockwise from $p$.

For a point $p\in P$, consider the family $\ff(p)\subset \ff$ of triangles with vertex $p$. We may assume that $\ff(p)$ is nonempty for any $p\in P$, otherwise we can simply apply induction to $P\setminus \{p\}$.  
Given a triangle  $pqr\in \ff(p)$, its side $qr$ can be identified with $A_{qr}\subset P\setminus \{p\}$. 
Let $\aaa(p) = \{A_{qr}: pqr\in \ff(p)\}$ be the collection of such arcs over all triangles in $\ff(p)$.  
Since $\ff$ is an intersecting family, each pair of arcs from $\aaa(p)$ 
share at least two consecutive  points from $P\setminus\{p\}$. Thus, there is an arc  $C(p)\subset P\setminus \{p\}$ consisting of at least two points, that is contained in  all arcs from $\aaa(p)$. (This follows from Helly's theorem in dimension one.)

\begin{cla}\label{lem1}
  There exists a pair of points $p,q\in P$, such that $p\in C(q)$ and $q\in C(p)$.
\end{cla}
\begin{proof}
Choose any triangle $pqr\in \ff$ for which  $|A_{pq}|$ is maximized. We show that $p$ and $q$ satisfy the assertion of the claim.  Consider a triangle $pq'r'\in \ff(p)$. If $q\notin A_{q'r'}$  then the clockwise order on these $5$ points is either $pq'r'q\ldots$ or $pqq'\ldots$. In the former case, the triangles $pqr$ and $pq'r'$ have disjoint interiors. In the latter case, we have $|A_{pq'}|>|A_{pq}|$. Both cases give a contradiction. This implies that $q\in A_{q'r'}$, and thus $q\in C(p)$. A symmetric argument shows that $p\in C(q)$. (See Figure \ref{fig2} for an illustration.)
\end{proof}

\begin{figure}
  \centering
\begin{tikzpicture}[scale = 0.5]

\begin{scope}[yscale = .85]
\fill[blue, opacity = .25] ({360*1/21}:4) -- ({360*8/21}:4) -- ({360*13/21}:4);
\fill[blue, opacity = .25] ({360*1/21}:4) -- ({360*15/21}:4) -- ({360*18/21}:4);

\foreach \x in {0,...,20}
\filldraw ({360*\x/21}:4) circle (3pt);

\node at ({360*8/21}:4.6) { $r$};
\node at ({360*1/21}:4.6) { $p$};
\node at ({360*13/21}:4.6) { $q$};
\node at ({360*18/21}:4.6) { $q'$};
\node at ({360*15/21}:4.6) { $r'$};
\end{scope}

\begin{scope}[xshift = 14cm, yscale = .85]
\fill[blue, opacity = .25] ({360*1/21}:4) -- ({360*8/21}:4) -- ({360*13/21}:4);
\fill[blue, opacity = .25] ({360*1/21}:4) -- ({360*6/21}:4) -- ({360*11/21}:4);

\foreach \x in {0,...,20}
\filldraw ({360*\x/21}:4) circle (3pt);

\node at ({360*8/21}:4.6) { $r$};
\node at ({360*1/21}:4.6) { $p$};
\node at ({360*13/21}:4.6) { $q$};
\node at ({360*11/21}:4.6) { $q'$};
\node at ({360*6/21}:4.6) { $r'$};
\end{scope}
\end{tikzpicture}
  \caption{Illustration for the proof of Claim~\ref{lem1}. On the left: the clockwise order $pq'r'q\dots$ gives triangles with disjoint interiors. On the right: the clockwise order $p,q,q'r'\dots$ gives $|A_{pq'}|> |A_{pq}|$.}\label{fig2}
\end{figure}
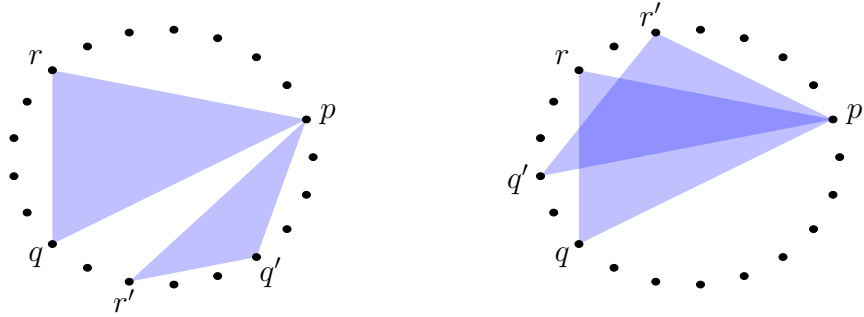

Consider a pair of points $p,q\in P$ guaranteed by Lemma~\ref{lem1}. Then $\ff(p)\cup \ff(q)= \ff_1\cup \ff_2\cup \ff_3$, where \begin{itemize}\item $\ff_1$ is the family of all triangles from $\ff$ with vertices $p$ and $q$;
\item $\ff_2$ ($\ff_3$) is the family of all triangles from $\ff(p)$ ($\ff(q)$), whose two vertices different from $p$ ($q$)  are separated by the line $\ell$ passing through $p$ and $q$.\end{itemize} Assume that $\ell$ contains $a$ points from $P$ on one side and $b$ points on the other, where $a+b = n-2$ and $a\le b$. We cannot simultaneously include in $\ff_1$ two triangles that have their third vertex on $A_{pq}$ and $A_{qp},$ respectively, because the interiors of these triangles do not intersect. Thus, $|\ff_1|\le b$. Similarly, for any pair of points $x,y$ lying on the opposite sides of the line connecting $p$ and $q$, the triangles $xyq$ and $xyp$ have disjoint interiors and thus only one of them may belong to $\ff$. Thus, $|\ff_2|+|\ff_3|\le ab$. Putting these bounds together and using $(a+1)+b = n-1,$
we get that \begin{equation}\label{eq9} |\ff(p)\cup \ff(q)|\le  (a+1)b\le \Big\lceil\frac{n-1}2\Big\rceil\Big\lfloor\frac{n-1}2\Big\rfloor.\end{equation}
It is easy to check for both odd and even $n$ that the last expression is equal to $\binom{\lceil \frac{n}{2} \rceil}{2} + \binom{\lfloor \frac{n}{2} \rfloor}{2}$. Indeed, for odd $n$ the right hand side in \eqref{eq9} is $(n-1)^2/4$, while $\binom{\lceil \frac{n}{2} \rceil}{2} + \binom{\lfloor \frac{n}{2} \rfloor}{2} = \frac{\frac{n+1}2\frac{n-1}2}2+\frac{\frac{n-1}2\frac{n-3}2}2 = (n-1)^2/4$. For even $n$  the right hand side in \eqref{eq9} is $\frac{n(n-2)}4$, while $\binom{\lceil \frac{n}{2} \rceil}{2} + \binom{\lfloor \frac{n}{2} \rfloor}{2} = 2{\frac n2\choose2} = \frac{n(n-2)}4$.
Applying inductive hypothesis to $P\setminus \{p,q\}$, we get that
\[|\ff| \leq \binom{\big\lceil \frac{n}{2} \big\rceil}{3} +
\binom{\big\lfloor \frac{n}{2} \big\rfloor}{3} +
\binom{\big\lceil \frac{n}{2} \big\rceil}{2} + \binom{\big\lfloor \frac{n}{2} \big\rfloor}{2} =
\binom{\big\lceil \frac{n+2}{2} \big\rceil}{3} + \binom{\big\lfloor \frac{n+2}{2} \big\rfloor}{3} = F(n).\]
This completes the proof.

\section{Open problems}
Let us first state some open questions concerning the case of intersecting families of triangles.
\begin{pro}
  What is the maximum and the minimum, over all point sets $P$ of size $n$, of the size of the largest intersecting family of triangles spanned by $P$?  Is the maximum always at most $\big(\frac 14 +o(1)\big) \binom{n}{3}$ as $n\to \infty$? Is the minimum 
at most $\big(c +o(1)\big) \binom{n}{3}$ with $c<1/4$ as $n\to \infty$? \end{pro}

\begin{pro}
  What happens if one relaxes the intersecting condition and allows triangles to intersect on the boundary?
\end{pro}
\begin{pro}
Find analogues of our results for other classes of sets such as convex $k$-gons  in $\R^2$.
\end{pro}
Finally, one may also ask similar questions in higher dimensions. For instance, given $n$ points in general position in $\mathbb{R}^d$ it is known (see for instance \cite[Remark 4.2]{uli}) that the number of $d$-simplices containing the origin in its interior is at most 
\[F_d(n) = \binom{\lfloor\frac{n+2}{2}\rfloor}{d+1} + \binom{\lceil\frac{n+2}{2}\rceil}{d+1}.\] 
\begin{pro}
  Let $P$ be a set of points in general position in $\R^d$  (or on the unit sphere in $\mathbb{R}^d$) and let $\ff$ be a family of $d$-simplices spanned by $P$ such that any $d$-tuple of them have a common point in their interior. Is it true that $|\ff| \leq F_d(n)$?
\end{pro}
Instead of requiring every $d$-tuple have a point in common interior point, it also makes sense to ask that every $t$-tuple have a common interior point for some fixed  $1<t\leq d$.\\

{\bf Remark. } Since the publication of the paper on ArXiv, there was an exciting progress on the problems mentioned above. First, F\"uredi et al. \cite{FM} resolved a part of Problem~1 which asks for the maximum size of an intersecting family of triangles for any set. They showed that the maximum is indeed at most $(\frac 14+o(1)){n\choose 3}$. They have also sketched a possible extension of these results to Problem 3. Second, Barnab\'as Janzer in an e-mail exchange provided a construction of a planar point set with the largest intersecting family having size at most $(c+o(1)){n\choose 3}$, providing an answer to another part of Problem 1. So far, the best possible value of c remains unknown.

\paragraph{Acknowledgement} This research was done while the first and third authors were visiting KAIST.

\end{document}